\documentclass{amsart}

\title{A relative bicommutant theorem: the stable case of Pedersen's question}

\author{Thierry Giordano}

\address{Department of Mathematics and Statistics\\
University of Ottawa\\
585 King Edward Ave.\\
Ottawa, Ontario\\
K1N 6N5\\
Canada}

\email{giordano@uOttawa.ca}

\author{Ping W. Ng}

\address{Department of Mathematics\\
University of Louisiana at Lafayette\\
217 Maxim Doucet Hall, P. O. Box 43568\\
Lafayette, Louisiana\\
70504--3568\\
USA}

\email{png@louisiana.edu}

\thanks{T. G. was partially supported by a grant
from NSERC Canada.}  

\newtheorem{InnerTh}{Theorem}
\newenvironment{Th}[1]
{\renewcommand\theInnerTh{#1}\InnerTh}
{\endInnerTh}

\newtheorem{thm}{Theorem}[section]
\newtheorem{df}[thm]{Definition}
\newtheorem{prop}[thm]{Proposition}
\newtheorem{cor}[thm]{Corollary}

\newtheorem{lem}[thm]{Lemma}

\newcommand{\A}{\mathcal{A}}
\newcommand{\B}{\mathcal{B}}
\newcommand{\C}{\mathcal{C}}
\newcommand{\D}{\mathcal{D}}
\newcommand{\E}{\mathcal{E}}

\newcommand{\K}{\mathcal{K}}

\newcommand{\hh}{\mathcal{H}}

\newcommand{\bb}{\mathbb{B}}

\newcommand{\Mul}{\mathcal{M}}

\newcommand{\Proj}{\textbf{Proj}}

\newcommand{\Lat}{\textbf{Lat}}

\newcommand{\Alg}{\mathcal{A}lg}

\begin{document}

\maketitle

\begin{abstract}
In 1976, D. Voiculescu proved that every separable unital sub-C*-algebra
of the Calkin algebra is equal to its (relative) bicommutant.  In his
minicourse (see \cite{PedersenCC}), G. Pedersen asked in 1988 if Voiculescu's
theorem can be extended to a simple corona algebra of a $\sigma$-unital
C*-algebra.   In this note, we answer Pedersen's question for a stable
$\sigma$-unital C*-algebra.
\end{abstract}

\section{Introduction}

Let $\hh$ be a separable infinite dimensional Hilbert space.
Recall that the Calkin algebra, the quotient of $\bb(\hh)$ by
the closed two-sided ideal $\K = \K(\hh)$ of compact operators, is a (non separable) simple C*-algebra.
In his ground-breaking publication on the noncommutative
Weyl--von Neumann theorem, D. Voiculescu proved the following:

\begin{Th}{A} 
(\cite{Voiculescu}, Corollary 19)
A separable unital sub-C*-algebra of the Calkin algebra is equal to its
(relative) bicommutant.
\label{Th:A}
\end{Th}

For a C*-algebra $\B$, let $\Mul(\B)$ denote its multiplier algebra and
$\Mul(\B)/\B$ its corona C*-algebra.
As the multiplier algebra of $\K$ is isomorphic to $\bb(\hh)$, the corona
C*-algebra of $\K$ is the Calkin algebra.

In his mini-course on the corona construction (\cite{PedersenCC}), G.K. Pedersen asks in Remark 10.11 if Voiculescu's
theorem can be extended to the case of a $\sigma$-unital algebra $\B$, whose
corona algebra is simple.

As the corona algebra of a $\sigma$-unital simple purely infinite C*-algebra
is simple (\cite{LinContScaleI}, \cite{LinSimpleCorona}), the following main result
of this note gives a positive answer to the stable case of Pedersen's question:

\begin{Th}{B}
Let $\B$ be a $\sigma$-unital stable simple and purely infinite C*-algebra.

Then any separable unital sub-C*-algebra of $\Mul(\B)/\B$ is equal to its
(relative) bicommutant.
\label{Th:B}
\end{Th}

This note is structured as follows. We quickly review in Section 2,  Elliott--Kucerovsky
theory of absorbing extensions and use it to show in Proposition \ref{prop:SimpleCorona} that a separable stable simple and purely infinite C*-algebra has a nice theory of extensions. This becomes one of the key tools of the proof of Theorem \ref{Th:B}.

Let $\A$ be a unital norm-closed subalgebra of a unital C*-algebra $\D$. Following
\cite{Voiculescu}, let $\Lat(\A)$ denote the set of all self-adjoint projections
$e$ of $\D$ such that $(1 - e)a e = 0$ for every $a \in \A$, and let
$\Alg(\Lat(\A))$  be the subalgebra of all $y \in \D$ such that $
 (1 -e) y e = 0, \forall e \in \Lat(\A).$
Notice that if $\D = \bb(\hh)$, then $\Lat(\A)$ is the lattice of closed invariant subspaces (projections) for $\A$, and
$\Alg(\Lat(\A)) = \{ T \in \bb(\hh)\;;\; \Lat(\A) \subset \Lat(T)\,\}$. Then $\A$ is a reflexive subalgebra of $\bb(\hh)$ if
$\A = \Alg(\Lat(\A))$.

In Theorem 1.8 of \cite{Voiculescu}, D. Voiculescu proves that if $\A$ is a separable unital norm-closed subalgebra of
the Calkin algebra, then $\A$ is equal to $\Alg(\Lat(\A))$.
The main result of Section 3 is Theorem \ref{thm:main1} where we generalizes Voiculescu's result and its proof by replacing $\K$ by any separable stable simple purely infinite C*-algebra $\B$. Corollary \ref{cor:main2} is then the direct generalization of Theorem \ref{Th:A}.

In Section 4, we generalize Theorem \ref{thm:main1} by relaxing the separability condition on the C*-algebra $\B$ and prove Theorem \ref{Th:B}
(as Corollary \ref{cor:main3}).

\section{Preliminaries in the theory of absorbing extensions}

The theory of extensions and of absorbing extensions has a long history
(see in particular \cite{BDFOriginal}, \cite{ElliottKucerovskyAbsorb},
\cite{KasparovAbsorbing},
\cite{LinAbsorbing2002}, \cite{Voiculescu}), culminating in the Elliott--Kucerovsky
theory of absorbing extensions.
In this section, we briefly sketch some preliminaries and recall
Elliott--Kucerovsky's Theorem 6.

  Using Lemma \ref{lem:AbsorpAndSimpleCorona}, we
summarize in
Proposition \ref{prop:SimpleCorona} several characterizations of a separable
stable simple purely infinite C*-algebra using properties of its theory
of extensions.

For  background notions in the theory of extensions, we use
monographs \cite{BlackadarBook} and \cite{ThomsenJensen}.

\begin{df} (See \cite{ElliottKucerovskyAbsorb} Definitions 1 and 2.)
\begin{enumerate}
\item \label{First} Let $\B$ be a C*-algebra and
$\E$  a C*-algebra containing $\B$ as a closed
two sided ideal.
We say $\E$ is \emph{purely large with respect to} $\B$ if for all
$d \in \E - \B$, the C*-algebra $\overline{d \B d^*}$ contains a
stable sub-C*-algebra which is full in $\B$.
\item Let $\A, \B$ be C*-algebras
and let
$$0 \rightarrow \B \rightarrow \E \rightarrow \A \rightarrow 0$$
be an extension of $\B$ by $\A$.  We say that the extension is
\emph{purely large} if the C*-algebra of the extension, $\E$, is
purely large with respect to the image of $\B$ in it, in the sense
described in (\ref{First}).
\end{enumerate}
\end{df}

Note that a purely large extension
$0 \rightarrow \B \overset{\gamma}{\rightarrow} \E
\overset{\pi}{\rightarrow} \A \rightarrow 0$ is essential
(that is, $\gamma(\B)$ is an essential closed two-sided ideal
of $\E$).

Recall that to any extension $0 \rightarrow \B \rightarrow \E \rightarrow \A \rightarrow 0$
is naturally associated a *-homomorphism from $\A = \E/\B$ to the corona algebra
$\Mul(\B)/\B$
called the \emph{Busby invariant} of the extension.
It is well-known
(see, for example, \cite{BlackadarBook}, 15.4)
that up to strong isomorphism (in the sense of Blackadar), an extension is
determined by its Busby invariant.
Therefore, we will not distinguish between extensions  and their Busby invariants.

For the rest of this section, let $\A$ be a unital C*-algebra.
Recall (\cite{BlackadarBook}, 15.5) that a unital,
extension $\phi : \A \rightarrow \Mul(\B)/\B$ is a
\emph{(strongly) unital trivial extension}
if it lifts to a unital *-homomorphism $\phi_0$ from $\A$ to $\Mul(\B)$.
To simplify notation and as in \cite{ElliottKucerovskyAbsorb},
we will omit the term strongly.

As in \cite{ElliottKucerovskyAbsorb}, Definition 5,
recall that a *-homomorphism $\phi_0 : \A \rightarrow \Mul(\B)$
is \emph{weakly nuclear} if
for each contraction $b \in \B$, the contractive c.p. map
$\A \rightarrow \B$ given by  $a \mapsto b \phi_0(a) b^*$ is nuclear (i.e.,
it factors approximately (in norm) through finite dimensional C*-algebras with
contractive c.p. maps).
Then a trivial unital extension $\phi : \A \rightarrow \Mul(\B)/\B$ is
\emph{weakly nuclear} if it lifts to a weakly nuclear unital *-homomorphism
$\phi_0 : \A \rightarrow \Mul(\B)$.

If $\B$ is stable, let $S$ and $T$ be two isometries of $\Mul(\B)$
such that  $S S^* + T T^* = 1_{\Mul(\B)}$.
Then for two extensions
$\phi, \psi : \A \rightarrow \Mul(\B)/\B$,  let $\phi \oplus \psi$
be their \emph{Brown--Douglas--Fillmore
(BDF) sum} given, for $a \in \A$, by
$$\phi \oplus \psi(a) = s \phi(a)s^* + t\psi(a) t^*$$
where $s = q(S)$ and $t = q(T)$
(and $q$ is the canonical surjection from $\Mul(\B)$ to
$\Mul(\B)/\B$).
The BDF sum is well-defined up to BDF-equivalence and is independent of the
choice of isometries $S$ and $T$ in $\Mul(\B)$.

Let $\A$ and $\B$ be separable C*-algebras with $\A$ unital and $\B$
stable.  A unital extension $\phi : \A \rightarrow \Mul(\B)/\B$ is
\emph{absorbing in the nuclear sense} if for any weakly nuclear (strongly)
unital trivial extension $\psi : \A \rightarrow \Mul(\B)/\B$, there exists a
unitary $U \in \Mul(\B)$ such that, for all $a \in \A$,
$$\phi(a) = u( \phi \oplus \psi(a))u^*$$
where $u = q(U)$.

Then the main result of Elliott--Kucerovsky's paper is the following:

\begin{thm}
(\cite{ElliottKucerovskyAbsorb}, Theorem 6)
Let $\A$ and $\B$ be separable C*-algebras with $\A$ unital and $\B$
stable.

For a unital essential extension $\phi : \A \rightarrow \Mul(\B)/\B$
of $\B$ by $\A$,
the following statements are equivalent:
\begin{enumerate}
\item $\phi$ is purely large.
\item
$\phi$ is absorbing in the nuclear sense.
\end{enumerate}
\label{thm:ElliottKucerovskyAbsorb}
\end{thm}

For the sake of completeness, we include a proof of the following lemma, even if
it is certainly well-known to the experts.

\begin{lem}
Let $\B$ be a separable stable simple C*-algebra with the following property:

For every unital separable C*-algebra $\A$, every unital essential extension of
$\B$ by $\A$ is absorbing in the nuclear sense.

Then the corona algebra $\Mul(\B)/B$ of $\B$ is simple.
\label{lem:AbsorpAndSimpleCorona}
\end{lem}

\begin{proof}

We prove the contrapositive.
Suppose that $\Mul(\B)/\B$ is not simple.
Let $a \in \Mul(\B)/\B$ be a nonzero self-adjoint element which is not
full in $\Mul(\B)/\B$, i.e., the closed two sided ideal generated by $a$
is a proper ideal of $\Mul(\B)/\B$.
Let $\phi : C^*(a, 1_{\Mul(\B)/\B}) \rightarrow \Mul(\B)/\B$ be the
inclusion map.  Then $\phi$ is a unital essential extension.  However,
$\phi$ is not absorbing in the nuclear sense,
since nuclearly absorbing extensions always bring nonzero
elements to full elements (and since $\phi(a)$ is not full).

\end{proof}

Recall that a nonunital $\sigma$-unital simple C*-algebra $\B$ has \emph{continuous scale}
if $\B$ has an approximate identity $\{ e_n \}_{n=1}^{\infty}$ such that
i. $e_{n+1} e_n = e_n$ for all $n$ and ii. for every $b \in \B_+ - \{ 0 \}$,
there exists $N \geq 1$ such that for all $m > n \geq N$,
$e_m - e_n \preceq b$.  (See \cite{LinContScaleI}.)

We summarize and recall in the next proposition several characterizations of the simplicity
of the corona algebra.  This also illustrates the correspondence between
``nice" extension theory and ``nice" corona algebra structure.

\begin{prop}
Let $\B$ be a separable stable simple C*-algebra.

Then  the following statements are equivalent:

\begin{enumerate}
\item Either $\B \cong \K$ or $\B$ is simple purely infinite.
\item Either $\B \cong \K$ or $\B$ has continuous scale.
\item $\Mul(\B)/\B$ is simple.
\item $\Mul(\B)/\B$ is simple purely infinite.
\item For every unital separable C*-algebra $\A$, every unital
essential extension of $\B$ by $\A$ is purely large.
\item For every unital separable C*-algebra $\A$, every unital
essential extension of $\B$ by $\A$ is absorbing in the nuclear
sense.
\end{enumerate}
\label{prop:SimpleCorona}
\end{prop}

\begin{proof}
The equivalence of (2), (3) and (4) can be found in
\cite{LinContScaleI} and \cite{LinSimpleCorona} (see also
\cite{Zhang1989}).

The equivalence of (1) and (3) can be found in
\cite{ZhangRiesz} and \cite{RorIdeal} (see also \cite{LinContScaleI}
and \cite{LinSimpleCorona}).

The equivalence of (5) and (6) is \cite{ElliottKucerovskyAbsorb} Theorem 6
(see our Theorem \ref{thm:ElliottKucerovskyAbsorb}).

That (1) implies (5) can be found in
\cite{ElliottKucerovskyAbsorb} Theorem 17 items (i) and (ii),
 as well as
\cite{Voiculescu}.

That (6) implies (3) is Lemma \ref{lem:AbsorpAndSimpleCorona}.

\end{proof}

We note that in Proposition \ref{prop:SimpleCorona} above,
statements (2), (3) and (4) are equivalent even without the
assumption of stability (e.g., see \cite{LinSimpleCorona}
Corollary 3.3).  We will require stability to relate simplicity of
the corona algebra to simple pure infiniteness of the canonical
ideal as well as nice extension theory.

\section{Generalization of Voiculescu's bicommutant theorem}

In this section, we first show three lemmas, which will be used in the
proof of Theorem \ref{thm:main1} and Corollary \ref{cor:main2}.
These two results
generalize Voiculescu's Theorem 1.8 and Corollary 1.9 of
\cite{Voiculescu}.   

Let us first recall  some notation introduced in
\cite{Voiculescu}.
If $\D$ is a unital C*-algebra and if $\A$ is a
subalgebra of $\D$, then $\Lat(\A)$ denotes the set of all projections
$e$ of $\D$ such that $(1 - e)a e = 0$ for every $a \in \A$, and let
$\Alg(\Lat(\A))$  be the subalgebra of $\D$ equal to
$$\Alg(\Lat(\A)) = \{ y \in \D : (1 -e) y e = 0, \forall e \in \Lat(\A) \}.$$
Note that, as remarked in \cite{ArvesonDuke} page 344,
$\Lat(\A)$ need not form a lattice, but we will follow the notation
of previous authors.

Notice that if $\Proj(\D)$ is the set of projections of $\D$ and
if $\A$ is a sub-C*-algebra of $\D$, then
$\Lat(\A) = \{ e \in \Proj(\D) : e \in \A' \}$ and
$(\A' \cap \D)' \cap \D \subseteq
\Alg(\Lat(\A))$.

In the proof of the last theorem (Theorem \ref{thm:Reflexive}),
in order to emphasize the ambient algebra $\D$,
we will write instead $\Lat_{\D}(\A)$ and
$\Alg_{\D}(\Lat_{\D}(\A))$.  For the rest of the paper, since
the ambient algebra is clear, we will drop the subscript ``$\D$".

\begin{lem}
Let $\D$ be a unital C*-algebra and $\A$ be a subalgebra of $\D$.
Then for any $d \in \D$,
$$\sup_{e \in \Lat(\A)} \| (1-e)d e \| \leq dist(d, \A).$$
\label{lem:Mar3020176AM}
\end{lem}

\begin{proof}

As for any  $a \in \A$ and $e \in \Lat(\A)$,
$$\| (1 - e) d e \| = \| (1 - e) (d - a) e \| \leq \| d - a \|,$$
the lemma is verified.
\end{proof}

The proof of the next lemma is contained in the first part of
the proof of
\cite{ArvesonDuke}, Corollary 2, pages 344 and 345.

\begin{lem}
Let $\C$ be a unital separable C*-algebra and
$\A$ be a  unital subalgebra of $\C$.

Then for any  $c \in \C$,
there exist a unital *-representation
$\psi : \C \rightarrow \mathbb{B}(\mathcal{H_{\psi}})$
on a separable infinite dimensional Hilbert space $\mathcal{H_{\psi}}$
and a projection $p \in \mathbb{B}(\mathcal{H_{\psi}})$
such that
\begin{enumerate}
\item $p \in \Lat(\psi(\A))$, i.e., $(1 -p)\psi(a) p = 0$
for all $a \in \A$,  and
\item $\| (1- p) \psi(c) p \| =
\| (1 - p) \psi(c) p \|_{ess} \geq dist(c, \A)$, where
the $\|. \|_{ess}$ denotes the essential norm.
\end{enumerate}
\label{lem:Mar3020177AM}
\end{lem}

\begin{lem}
Let $\B$ be a separable stable C*-algebra, and let $\A$ be a separable
norm-closed unital subalgebra of the corona algebra $\Mul(\B)/\B$.

If
$x \in \Mul(\B)/\B$ is an element such that
the inclusion of the C*-algebra  $C^*(x, \A)$ generated by $x$ and $\A$
in $\Mul(\B)/\B$ is
purely large,
then there exists
a projection $p \in \Mul(\B)/\B$ such that
\begin{enumerate}
\item $p \in \Lat(\A)$, and
\item $\|(1 -p) x p \| = dist(x, \A).$
\end{enumerate}
\label{lem:Mar3020178AM}
\end{lem}

\begin{proof}
By Lemma \ref{lem:Mar3020176AM}, it suffices to find
a projection $p \in \Lat(\A)$ such that
$$\|(1 -p) x p \| \geq dist(x, \A).$$

Since $\B$ is stable, we may replace $\B$ by $\B \otimes \K$, where
$\K$ is the C*-algebra of compact operators on a separable infinite
dimensional Hilbert space $\hh$.

Let us denote by $\C$ the separable unital sub-C*-algebra of
$\Mul(\B \otimes \K)/(\B \otimes \K)$ generated by $x$ and $\A$.

Since $\Mul(\K) = \hh$,
let us denote by  $\sigma : \C \rightarrow \Mul(\K)$ the
unital *-homomorphism
given by Lemma \ref{lem:Mar3020177AM},
and by  $q_0 \in \Mul(\K)$ the projection such that
\begin{enumerate}
\item[i.] $(1 -q_0) \sigma(a) q_0 = 0$, for all $a \in \A$, and
\item[ii.] $\| (1 - q_0) \sigma(x) q_0 \| =
\| (1 - q_0) \sigma(x) q_0 + \K \|_{\Mul(\K)/\K}
\geq dist(x, \A)$.
\end{enumerate}

Note that since $\K$ is  nuclear,
the unital trivial extension
\begin{equation}
\C \rightarrow \Mul(\B \otimes \K)/\B \otimes \K : c \mapsto
\pi (1_{\Mul(\B)} \otimes \sigma(c))
\label{equ:Oct182017}
\end{equation}
is weakly nuclear, where $\pi$ denotes the canonical surjection from
$\Mul(\B \otimes \K)$ onto the corona algebra $\Mul(\B \otimes \K)/ \B \otimes \K$.

For any two isometries $S, T \in \Mul(\B \otimes \K)$  such that
$1_{\Mul(\B \otimes \K)} = S S^* + T T^*$, let
$s = \pi(S)$ and
$t = \pi(T)$  denote the corresponding
two isometries of $\Mul(\B \otimes \K)/\B \otimes \K$.

As by assumption the inclusion of $\C$ in $\Mul(\B \otimes \K)/\B \otimes \K$
is purely large, and as the trivial extension (\ref{equ:Oct182017})
is weakly nuclear, then by Theorem
\ref{thm:ElliottKucerovskyAbsorb},
there exists a unitary
$U \in \Mul(\B \otimes \K)$ such that
for all $c \in \C$,
$$c = u (s c s^* + t \pi (1_{\Mul(\B)}
\otimes \sigma(c)) t^*) u^*$$
where $u = \pi(U)$.

Let $p$ be the projection of $\Mul(\B \otimes \K)/(\B \otimes \K)$
given by
$$p =_{df} u t \pi (1_{\Mul(\B)} \otimes q_0) t^* u^*,$$
and let us
show that $p$ is the required projection.

As $(1-p) a p = u t \pi(1_{\Mul(\B)} \otimes (1_{\Mul(\K)} - q_0) \sigma(a) q_0)
t^* u^* = 0$ for all $a \in \A$,
we have that $p \in \Lat(\A)$.

Moreover,
\begin{eqnarray*}
& & \| (1 - p) x p \| \\
& = & \| (1 - p) u t \pi(1_{\Mul(\B)} \otimes \sigma(x)) t^* u^* u t \pi(1_{\Mul(\B)}
\otimes q_0) t^* u^* \| \\
& = & \| (1 - p) u t \pi(1_{\Mul(\B)} \otimes \sigma(x) q_0)  \| \\
& = & \| [us s^*u^* + u t \pi(1_{\Mul(\B)} \otimes (1_{\Mul(\K)}- q_0)) t^* u^*]u t
\pi(1_{\Mul(\B)} \otimes \sigma(x) q_0)  \| \\
& = & \| u t \pi(1_{\Mul(\B)} \otimes (1_{\Mul(\K)} - q_0)\sigma(x) q_0)  \| \\
& = & \| (1 - q_0) \sigma(x) q_0 + \K \|_{\Mul(\K)/\K} \\
& \geq & dist(x, \A) \makebox{    (by ii).} \\
\end{eqnarray*}

\end{proof}

\begin{thm}
Suppose that either $\B \cong \K$ or
$\B$ is a separable simple stable purely infinite C*-algebra.
Let $\A$ be a norm-closed separable unital subalgebra of
the corona algebra $\Mul(\B)/\B$, and $\Lat(\A)$ be the
set of all projections $e \in \Mul(\B)/\B$ such that
$(1 - e) a e = 0$, for all $a \in \A$.

Then the algebra of all elements of $\Mul(\B)/\B$ leaving
invariant each projection of $\Lat(\A)$ is equal to $\A$.


\label{thm:main1}
\end{thm}

As for any sub-C*-algebra $\A$ of
a unital C*-algebra $\D$, we have

$$\A \subseteq (\A' \cap \Mul(\B)/\B)' \cap \Mul(\B)/\B
\subseteq \Alg(\Lat(\A)), $$
the following corollary is a direct consequence of Theorem \ref{thm:main1}:

\begin{cor}
If $\B \cong \K$ or
$\B$ is a separable simple stable purely infinite C*-algebra,
and $\A$ is a separable unital sub-C*-algebra of $\Mul(\B)/\B$,
then $\A$ is equal to its relative bicommutant, i.e.,
$$\A = (\A' \cap \Mul(\B)/\B)' \cap \Mul(\B)/\B.$$
\label{cor:main2}
\end{cor}


\begin{proof}[Proof of Theorem \ref{thm:main1}]
Let $$\Alg(\Lat(\A)) =_{df}
\{ x \in \Mul(\B)/\B : (1 - e)x e = 0, \forall e \in
\Lat(\A) \}$$
denote the algebra of elements of $\Mul(\B)/\B$ leaving
invariant $\Lat(\A)$.

As $\A$ is clearly contained in $\Alg(\Lat(\A))$, let us
prove the other inclusion.

Let $x \in \Alg(\Lat(\A))$ and let
$\C = C^*(x, \A)$ be the separable unital sub-C*-algebra
of $\Mul(\B)/\B$ generated by $x$ and $\A$.
By Proposition \ref{prop:SimpleCorona},
the inclusion of
$\C$ into  $\Mul(\B)/\B$ is purely large.
Hence,
by Lemma \ref{lem:Mar3020178AM}, there exists a projection
$p \in \Lat(\A) \subset \Mul(\B)/\B$ such that
$$\| (1-p) x p \| = dist(x, \A).$$

As $x \in \Alg(\Lat(\A))$ and
$\A$ is closed, $x \in \A$.
\end{proof}

\section{Pedersen's question}

Theorem \ref{thm:Reflexive}, the main result of this section generalizes
Theorem \ref{thm:main1}, by relaxing the separability condition on
the C*-algebra $\B$.  Then in Corollary \ref{cor:main3},
we give a complete answer in the stable case to Pedersen's question
in \cite{PedersenCC}.

Before stating and proving these results, let us first introduce some
preliminaries needed below.

\begin{lem}
Let $\B$ be a simple, purely infinite C*-algebra, and let
$a, b \in \B_+$, $\| a \| = \| b \| = 1$.

Then for any $0 < \delta < 1$, and for any $\epsilon > 0$,
there exists $x \in \B$ such that
$$\| a - x b x^* \| \leq \epsilon \makebox{  and  }
\| x \|^2 \leq \frac{1}{\delta}.$$
\label{lem:CuntzWithControl}
\end{lem}

\begin{proof}
For $0 < \delta < 1$, fixed, we have $\| b - (b - \delta)_+ \| \leq \delta$
and as $\B$ is simple purely infinite, $a \preceq (b - \delta)_+$.
Then for any $\epsilon > 0$,
there exists by \cite{KaftalNgZhang}, Lemma 2.2,
$x \in \B$ with $\| x \|^2 \leq \frac{1}{\delta}$ and
$(a - \epsilon)_+ = xb x^*$.
Then
$\| a - x b x^* \| \leq \| a - (a - \epsilon)_+ \| \leq \epsilon$.
\end{proof}

Let $\C$ be a separable sub-C*-algebra of a $\sigma$-unital, stable,
simple, purely infinite C*-algebra $\B$,
and let $\{ c_n \}_{n \geq 1}$ be a countable dense subset of the unit
sphere of $\C_+$.

By Lemma \ref{lem:CuntzWithControl},
for each $n, m$ and $k \geq 1$, there exists $y_{n,m, k} \in \B$, such that
$\| y_{n,m, k} \|^2 \leq \frac{3}{2}$ and
$$\| y_{n,m,k} c_n y_{n,m,k}^* - c_m \| < \frac{1}{k}.$$

Then the separable sub-C*-algebra $\widetilde{\C}$ of $\B$ generated by
$\C$ and $\{ y_{n,m,k} : n, m, k \geq 1 \}$ is said to be obtained from
$\C$ by a \emph{PI-operation}.

 Using this construction inductively, we have the following:

\begin{prop}
Let  $\B$ be a $\sigma$-unital, stable, simple purely infinite C*-algebra,
and let $\C \subseteq \B$ be a separable sub-C*-algebra.

Then there exists
a separable, stable, simple, purely infinite sub-C*-algebra $\D$ of $\B$
containing $\C$.
\label{prop:SeparableSubalgebra}
\end{prop}

\begin{proof}
Let $b \in \B_+$ be a strictly positive element, and let
$\C_0$ be the (separable) sub-C*-algebra generated by $\C$ and $b$.\
By \cite{RordamStableSurvey} Corollary 2.3,
any sub-C*-algebra of $\B$ containing $\C_0$ is stable.
Then construct an increasing sequence $\{ \C_l \}_{l \geq 0 }$ of
separable sub-C*-algebras of $\B$ such that $\C_{l+1}$ is obtained from
$\C_l$ by a PI-operation, and
$\D =_{df} \overline{\bigcup_{l \geq 1} \C_l }$.
Then $\D$ contains $\C_0$, is stable, and simple (by similar argument as in
\cite{BlackadarGeneralBook} II.8.5.6).

To prove that $\D$ is purely infinite, it is enough to show that for any
$c, d \in \D_+$, $\| c \| = \| d \| = 1$, for any $\epsilon > 0$, there exists
$y \in \D$ such that
$$\| y c y^* - d \| < \epsilon.$$

For $l \geq 0$, let $\{ c^l_n \}_{n \geq 1}$ denote the dense sequence of the
unit sphere of $(\C_l)_+$, used in the construction of $\C_{l+1}$.
Then choose $l \geq 1$ and $c^l_n, c^l_m \in (\C_l)_+$ such that
$$\| c - c^l_n \| \leq \frac{\epsilon}{4}  \makebox{  and  }
\| d - c^l_m \| \leq \frac{\epsilon}{3}.$$

By definition of $\C_{l+1}$, there exists a $y \in \C_{l+1}$,
$\| y \|^2 \leq \frac{3}{2}$ such that
$$\| c^l_m - y c^l_n y^* \| \leq \frac{\epsilon}{3}.$$

Hence,
$$\| d - y c y^* \| \leq \frac{\epsilon}{3} + \frac{\epsilon}{3}
+ \| y (c^l_n - c) y^* \|
\leq \frac{\epsilon}{3} + \frac{\epsilon}{3} + \frac{3}{2} * \frac{\epsilon}{4}
\leq \epsilon.$$

\end{proof}

\begin{thm}
Let $\B$ be either the compact operators $\K$ or a $\sigma$-unital,
stable, simple, purely infinite C*-algebra.

Let $\A$ be a separable, norm-closed, unital subalgebra of $\Mul(\B)/\B$.

Then $$\A = \Alg(\Lat(\A)).$$
\label{thm:Reflexive}
\end{thm}

Recall (\cite{LinContScaleI} Theorem 3.8)
that if $\B$ is a  stable $\sigma$-unital C*-algebra,
then the corona algebra $\Mul(\B)/\B$ is simple if and only if either
$\B \cong \K$ or $\B$ is simple purely infinite.

Then the following corollary, which follows immediately from
Theorem \ref{thm:Reflexive}, resolves the stable case of Pedersen's
question.

\begin{cor}
Suppose that either $\B \cong \K$ or  $\B$ is
 a $\sigma$-unital stable  simple purely infinite C*-algebra.

Let $\A$ be a separable unital sub-C*-algebra of $\Mul(\B)/\B$.

Then $$(\A' \cap \Mul(\B)/\B)' \cap \Mul(\B)/\B = \A.$$
\label{cor:main3}
\end{cor}

In the proof of Theorem \ref{thm:Reflexive}, we will need the following
notions and results from \cite{KaftalNgZhang}:

Let $\B$ be a $\sigma$-unital simple C*-algebra.  Then a sequence
$\{ x_l \}_{l \geq 1}$ of $\B_+$ gives a \emph{bidiagonal series}
$X$ if $X = \sum_{l \geq 1} x_l$ converges in the strict topology and
$x_k x_l = 0$ for all $|k - l| \geq 2$.

If $T \in \Mul(\B)_+$, then by \cite{KaftalNgZhang}, Theorem 4.2,
for any $\epsilon > 0$, there exist a bidiagonal series
$\sum_{k \geq 1} t_k$ and a self-adjoint element $a_{\epsilon} \in \B$ with
$\| a_{\epsilon} \| < \epsilon$ and
$$T = \sum_{k \geq 1} t_k + a_{\epsilon}.$$

\begin{proof}[Proof of Theorem \ref{thm:Reflexive}]
To simplify notation, let us denote $\C(\B) =_{df} \Mul(\B)/\B$.

As $\A$ is already contained in $\Alg_{\C(\B)}(\Lat_{\C(\B)}(\A))$,
let us prove that
if
$x \in \Alg_{\C(\B)}(\Lat_{\C(\B)}(A))$  then $x \in \A$.

Let $\{ a_l \}_{l=1}^{\infty}$ be a dense sequence in $\A$
and choose $X$ and $A_l$ in $\Mul(\B)$ such that
$\pi(X) = x$ and $\pi(A_l) = a_l$ for $l \geq 1$ (where $\pi : \Mul(\B)
\rightarrow \Mul(\B)/\B$ is the canonical quotient map).
By \cite{KaftalNgZhang}, Theorem 4.2,
there exist for $l \geq 1$ and $1 \leq j \leq 4$,
bidiagonal series $X_j$ and $A_{l,j}$ in $\Mul(\B)_+$
and $y, c_l \in \B$,
such that
$$X = \sum_{j=1}^4 i^j X_j + y \makebox{  and  }
A_l = \sum_{j=1}^4 i^j A_{l,j} + c_l,$$
Then, let us write
$X_j = \sum_{n \geq 1} x_{j,n}$ and $A_{l,j} = \sum_{n \geq 1} d_{l,j,n}$,
where for $1 \leq j \leq 4$ and $l \geq 1$,
$\{ x_{j,n}, d_{l,j,n} : n \geq 1 \}
\subset \B_+$  and  the
series  converge in the  strict topology.

Let $b \in \B_+$ be strictly positive.
Let $\C$ denote the
separable sub-C*-algebra of $\B$, generated by
$\{b, y, c_l, x_{j,n}, d_{l,j,n} : 1 \leq j \leq 4, l, n\geq 1,  \}$.
By Proposition \ref{prop:SeparableSubalgebra},
let $\D$ be a separable, stable, simple, purely infinite sub-C*-algebra
of $\B$ containing $\C$.

For $l \geq 1$ and $1 \leq j \leq 4$, by construction,
$X_j$ and $A_{l,j}$ belong to $\Mul(\D)$, and therefore,
$X, A_l \in \Mul(\D)$.

Since $b \in \D$,
the inclusion  $\D \subseteq \B$ induces the
(unital) inclusion
$\Mul(\D) \subseteq \Mul(\B)$ and $\Mul(\D)/\D \subseteq
\Mul(\B)/\B$.
Moreover, $x = \pi(X)$ and $a_l = \pi(A_l)$, for $l \geq 1$, belong
to $\Mul(\D)/\D$.

Since $x \in \Alg_{\C(\B)}(\Lat_{\C(\B)}(\A)) \cap \C(\D)$,
$x \in \Alg_{\C(\D)}(\Lat_{\C(\D)}(\A))$ and by
Theorem \ref{thm:main1},
$x \in \A$.
\end{proof}

Farah asked whether $(A' \cap \D^{**})' \cap \D = \A$ for every unital
simple purely infinite separable C*-algebra $\D$ and $\A$
any unital
sub-C*-algebra of $\D$ (\cite{FarahBicommutant}, Question 5.2).
We note that, by arguments similar to those for Proposition
\ref{prop:SeparableSubalgebra} and Theorem \ref{thm:Reflexive},
this question is equivalent to the same question but without the
requirement that $\D$ be separable.
We have actually answer Farah's
question in the affirmative for the case where $\D$ is a nonseparable
simple purely infinite
C*-algebra of the form $\Mul(\B)/\B$ where $\B$ is $\sigma$-unital and
stable.
Here is the short argument:  If $x \in (\A' \cap (\Mul(\B)/\B)^{**})'
\cap \Mul(\B)/\B$, then by \cite{FarahBicommutant}, Proposition 5.3,
$x \in (\A' \cap (\Mul(\B)/\B))'
\cap \Mul(\B)/\B$, and so by
Corollary \ref{cor:main3}, $x \in \A$.

\end{document}